\documentclass[10pt,reqno]{amsart}

\usepackage[T1]{fontenc}
\usepackage{lmodern}
\usepackage{microtype} 
\usepackage{setspace}
\usepackage[usenames,dvipsnames]{xcolor}
\usepackage{tikz}
\usetikzlibrary{positioning, arrows.meta}
\usepackage{comment}

\usepackage[margin=1.25in]{geometry} 
\usepackage{mathtools}
\usepackage{amsmath,amsfonts,amsbsy,amsgen,amscd,mathrsfs,amssymb,amsthm}
\usepackage[colorlinks=true,citecolor=blue,linkcolor=blue,urlcolor=blue]{hyperref}
\usepackage{prettyref}
\newrefformat{sec}{Section~\ref{#1}}
\newrefformat{conj}{Conjecture~\ref{#1}}
\newrefformat{thm}{Theorem~\ref{#1}}
\newrefformat{lem}{Lemma~\ref{#1}}
\newrefformat{def}{Definition~\ref{#1}}
\newrefformat{eq}{(\ref{#1})}

\makeatletter
\renewcommand\paragraph{\@startsection{paragraph}{4}{\z@}%
  {1.5ex \@plus .2ex \@minus .2ex}%
  {-1em}%
  {\normalfont\normalsize\bfseries}}
\makeatother

\newtheorem{theorem}{Theorem}[section]
\newtheorem{lemma}[theorem]{Lemma}
\newtheorem{conjecture}{Conjecture}

\theoremstyle{definition}
\newtheorem{definition}[theorem]{Definition}
\theoremstyle{remark}
\newtheorem{remark}[theorem]{Remark}

\numberwithin{equation}{section}

\newcommand{\R}{\mathbb{R}}
\newcommand{\tr}{\text{tr} }

\newcommand{\norm}[1]{\left\| #1 \right\|}

\begin{document}

\title[On Brouwer's Laplacian conjecture]{On Brouwer's Laplacian conjecture}

\author[Kothari]{Pravesh K. Kothari}
\author[Tudose]{Stefan Tudose}

\address{Department of Computer Science, Princeton University, Princeton, NJ 08544, USA}
\email{kothari@cs.princeton.edu}

\address{Department of Mathematics, Princeton University, Princeton, NJ 08544, USA}
\email{studose@princeton.edu}

\begin{abstract}
Brouwer's Laplacian conjecture states that the sum of the largest $k$ eigenvalues of a graph's Laplacian is less than or equal to the number of edges plus $\binom{k+1}{2}$. We give a proof of this conjecture. Our proof relies on the Grone--Merris--Bai theorem for \emph{split} graphs. We also show the converse implication, giving an equivalence between Brouwer's conjecture and the Grone--Merris--Bai theorem. 
\end{abstract}

\maketitle

\section{Introduction}
\label{sec:intro}
Let $G$ be a simple graph on $n$ vertices and $m$ edges. Let $A$ be its $n \times n$ adjacency matrix and let $D$ be the diagonal matrix with $D(i,i) = d_i$, the degree of vertex $i$. The Laplacian of $G$ is the positive semidefinite matrix $L = D-A$. In the following, we write $\lambda_1\ge \lambda_2\ge \ldots\ge \lambda_n=0$, for the eigenvalues of the Laplacian of a graph $G$. 

Brouwer conjectured \cite{BrouwerHaemers2012,HaemersMohammadianTayfehRezaie2010} the following bound on the sum of top $k$ eigenvalues of the Laplacian of a graph $G$:
\begin{conjecture}[Brouwer's Laplacian conjecture] \label{conj:brouwer}
    The following inequality holds:
    $$\lambda_1+\ldots+\lambda_k\le m+\dbinom{k+1}{2},$$
    for any $1\le k\le n$. 
\end{conjecture}

Brouwer's conjecture is a basic problem in spectral graph theory and has attracted substantial attention since its formulation by Brouwer and Haemers \cite{BrouwerHaemers2012,HaemersMohammadianTayfehRezaie2010}. There are two main kinds of results known about the problem. The first kind proves the conjecture for special classes of graphs\cite{HaemersMohammadianTayfehRezaie2010,FritscherHoppenRochaTrevisan2011,DuZhou2012,WangHuangLiu2012,Mayank2010,HelmbergTrevisan2017,TorresTrevisan2024} or under additional structural assumptions such as bounds on girth, vertex cover, diameter and other such graph parameters~\cite{GanieAlghamdiPirzada2016,Chen2018,Chen2019,GaniePirzadaRatherTrevisan2020,Cooper2021}. The second line of work proves \emph{approximate} results that hold for every graph\cite{Cooper2021,LewPartitionDensity2024,LewSumsDegrees2025,LewApproximate2026}. Recent work has also studied closely related Kikuchi-graph and token-graph viewpoints, with applications to Brouwer-type eigenvalue bounds and bounds on the ground state energy of certain quantum Hamiltonians~\cite{LewApproximate2026,BakshiBasuKothariLi2026}. 

Brouwer was motivated by a related conjecture of Grone and Merris about the relationship between the Laplacian spectrum of a graph and its degree sequence. The main observation is that for \emph{threshold} graphs, the Laplacian eigenvalues, written in decreasing order, are exactly equal to the \emph{conjugate degree sequence}~\cite{GroneMerrisSunder1990,Merris1994,HammerKelmans1996}. 

\begin{definition}[Conjugate Degree Sequence]
For a graph $G$, let $d_k'$ be the number of vertices in $G$ of degree at least $k$. Then, $(d_1', d_2', \cdots, d_n')$ is called the \emph{conjugate} degree sequence of $G$. 
\end{definition}

Grone and Merris conjectured that the threshold-graph comparison is an extreme case and that in fact for any graph the Laplacian eigenvalue sequence is \emph{majorized} by the conjugate degree sequence \cite{GroneMerris1994}. Recall that a sequence $a=(a_1,\ldots, a_n)$ is majorized by a sequence $b=(b_1,\ldots, b_n)$ if the inequality  $\displaystyle\sum_{i=1}^j a_i \leq \sum_{i=1}^j b_i$  is true for any $1\le j\le n-1$ and $a_1+\ldots+a_n=b_1+\ldots+b_n$.

\begin{theorem}[Grone-Merris-Bai]
The following inequality holds
$$\lambda_1+\ldots+\lambda_k\le d_1'+\ldots+d_k',$$
for any $1\le k\le n$. 
\end{theorem}

The Grone-Merris conjecture was proved in full by Bai \cite{Bai}. Before Bai's proof, Stephen proved several special cases~\cite{Stephen2007}, including all trees, and gave an equivalent formulation in terms of Dirichlet Laplacians. Later work of Abiad, Fiol, Haemers, and Perarnau~\cite{AbiadFiolHaemersPerarnau2014} developed related eigenvalue interlacing bounds for sums of the largest Laplacian eigenvalues. We note that the Grone-Merris conjecture has also inspired analogs for simplicial complexes, notably, the Duval--Reiner conjecture \cite{DuvalReiner2002}.

In this paper, we prove Brouwer's conjecture using the Grone-Merris-Bai theorem. We also show a simple proof of the Grone-Merris-Bai theorem from Brouwer's conjecture showing that the two statements are in a sense equivalent. Curiously, our proof of Brouwer's theorem invokes the Grone-Merris bound for ``split graphs'' that also arise as the main technical case in the proof of~\cite{Bai}.

\section{Preliminaries} Our proof of Brouwer's conjecture is based on a nuclear norm bound on ``centered'' Laplacian of \emph{split graphs}. Let us recall some relevant definitions and set up some basic notation. 

\begin{definition}[Split Graph]
A graph $G$ on vertex set $[n]$ is a \emph{split graph} if there is a subset $K \subseteq [n]$ such that the induced subgraph on $K$ is a clique and the induced subgraph on $S=[n] \setminus K$ is an independent set.
\end{definition}

\textbf{Notation:} Let $J=11^T$ be the all ones matrix and  $C=I_n-\frac{1}{n}J$ be the centering matrix. We denote by $\norm{A}_*$ the nuclear norm of the matrix $A$. For a graph $G$ and disjoint vertex sets $K,S$, we write $E(K,S)$ to denote the number of ``cut'' edges with one end point in $K$ and $S$ each.

\section{Results and Proof Plan}

We give a proof of Brouwer's conjecture in this work. 

\begin{theorem}[Brouwer's Conjecture]\label{thm:brouwer}
For any graph $G$ with Laplacian eigenvalues $\lambda_1 \geq \lambda_2 \geq \cdots \lambda_n\geq 0$, we have:
    $$\lambda_1+\ldots+\lambda_k\le m+\dbinom{k+1}{2},$$
    for any $1\le k\le n$.
\end{theorem}

Our proof relies on a bound on the nuclear norm of the ``centered'' Laplacians of split graphs formulated below.    

\begin{lemma}[Centered Split Graph Nuclear Norm Bound]\label{lem:nuclearnormresult}
    Let $H$ be a split graph which has a clique $K$ and an independent set $S$. The following inequality holds
    $$\norm{L_H-|K|C}_*\le |K|\cdot |S|.$$
\end{lemma}

We will prove \prettyref{lem:nuclearnormresult} using the Grone-Merris-Bai theorem. We need the Grone-Merris-Bai theorem only for the case of split graphs but note that split graphs are the main ``hard'' case in Bai's proof of the Grone-Merris conjecture.

\paragraph{From Brouwer to Grone-Merris-Bai:} We next show that Brouwer's conjecture implies the following trace inequality for all split graphs.  

\begin{lemma}[Split Graph Trace Inequality]\label{lem:traceresult}
    Let $H$ be a split graph which has a clique $K$ and an independent set $S$. For any orthogonal projection $Q$ in $\R^n$ with $Q1=0$ the following inequality holds
    $$\mathrm{\tr} \left (Q(L_H-|K|I_n)\right )\le E(K,S)$$
\end{lemma}

This inequality easily allows us to recover a proof of the Grone-Merris-Bai theorem.

Here is the structure of the paper:
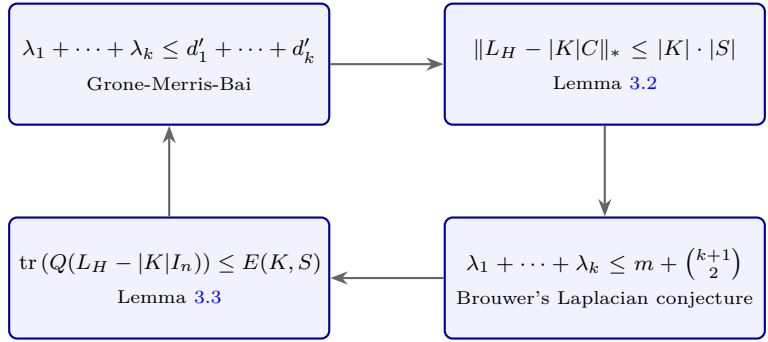
\begin{center} 
\begin{tikzpicture}[
    mybox/.style={
        rectangle,
        draw=blue!60!black,         
        fill=blue!5!white,          
        rounded corners=3pt,
        thick,
        minimum width=3.8cm,        
        minimum height=1.6cm,       
        text width=4cm,           
        align=center,
        font=\footnotesize,         
        node distance=1.2cm and 1.5cm 
    },
    arrow/.style={
        -{Stealth[scale=1.0]},      
        thick,
        draw=gray!80!black          
    }
]

    
    \node[mybox] (NW) {
        $\lambda_1+\dots+\lambda_k \le d_1'+\dots+d_k'$\\[0.3em]
        \scriptsize Grone-Merris-Bai
    };
    
    \node[mybox, right=of NW] (NE) {
         $\|L_H-|K|C\|_* \le |K|\cdot |S|$\\[0.3em]
        \scriptsize \prettyref{lem:nuclearnormresult}
    };
    
    \node[mybox, below=of NW] (SW) {
        $\mathrm{tr} \left(Q(L_H-|K|I_n)\right) \le E(K,S)$\\[0.3em]
        \scriptsize \prettyref{lem:traceresult}
    };
    
    \node[mybox, right=of SW] (SE) {
        $\lambda_1+\dots+\lambda_k \le m+\binom{k+1}{2}$\\[0.3em]
        \scriptsize Brouwer's Laplacian conjecture
    };

    \draw[arrow] (NW) -- (NE);  
    \draw[arrow] (NE) -- (SE);  
    \draw[arrow] (SE) -- (SW);  
    \draw[arrow] (SW) -- (NW);  

\end{tikzpicture}
\end{center}

\section{All but Brouwer}
We defer the proof that \prettyref{lem:nuclearnormresult} implies Brouwer's conjecture to the next section and present simple proofs of the remaining implications in this section.

\begin{proof}[Grone-Merris-Bai implies \prettyref{lem:nuclearnormresult}]
    Let $r=|K|,\ s=|S|,\ n=r+s$ and let $\lambda_1\ge \ldots\ge\lambda_n$ be the eigenvalues of the Laplacian $L_H$ of $H$. 
    By Grone-Merris-Bai, for any $1\le k\le n$ we have the inequality
    $$\sum\limits_{i=1}^{k}\lambda_i\le \sum\limits_{i=1}^{k}d_i'=\sum\limits_{i=1}^{n}\min(d_i,k)=\sum\limits_{i\in K}\min(d_i,k)+\sum\limits_{j \in S}\min(d_j,k)\le kr+E(K,S).$$
    The identity $\sum\limits_{i=1}^{k}d_i'=\sum\limits_{i=1}^{n}\min(d_i,k)$ follows by double counting the pairs $(i,j)$ with $ d_i\ge j,\ 1\le j\le k$.  

    Let \(t=\#\{i\le n-1:\lambda_i>r\}\). If \(t=0\), then
\(\sum_{\lambda_i>r}(\lambda_i-r)=0\). If \(t\ge 1\), the preceding estimate
with \(k=t\) gives
\[
\sum_{\lambda_i>r}(\lambda_i-r)
=
\sum_{i=1}^t \lambda_i-tr
\le E(K,S).
\]
    
    Since $1$ is an eigenvector with eigenvalue $0$ for both $L_H$ and $C$, we obtain
    $$\norm{L_H-rC}_*=\sum\limits_{i=1}^{n-1}|\lambda_i-r|=2\sum\limits_{\lambda_i>r}(\lambda_i-r)+rs-2E(K,S)\le rs.$$

\end{proof}

\begin{proof}[Brouwer's Laplacian conjecture implies \prettyref{lem:traceresult}]
    Let $r=|K|,\ s=|S|,\ n=r+s$ and $q=\text{rank}\ Q$.  Define $\overline{Q}=C-Q$ and let $\overline{H}$ be the complement graph of $H$. We use the complement to reduce to the case $\operatorname{rank} Q<r$. 

    Since $Q1=0$, we have $Q\preceq  C$, so $\overline Q=C-Q$ is again an
orthogonal projection and $\overline Q1=0$. Moreover $\operatorname{rank}\overline Q=n-1-q$.
If $q\ge r$, then
$\operatorname{rank}\overline Q=n-1-q\le n-1-r=s-1<s$,
so the complementary split graph has clique size \(s\) and falls into the
case $\operatorname{rank}\overline Q<s$.

    Now, observe that $\tr \left ( \overline{Q}(L_{\overline{H}}-sI_n)\right )=\tr \left (Q(L_H-rI_n)\right )+rs-2E(K,S)$. Thus, 
    $$\tr \left (Q(L_H-rI_n)\right )\le E(K,S) \Leftrightarrow \tr \left ( \overline{Q}(L_{\overline{H}}-sI_n)\right )\le rs-E(K,S).$$
    We can therefore assume without loss of generality that $\text{rank} \ Q=q<r$ by working with $(\overline{Q},\overline{H})$ instead of  $(Q,H)$ if necessary.
    
    Define the subspace $\mathcal{K}=\{x\in \R^n: \ \langle x,1_K\rangle =0,\ x|_{S}=0\}$. This is an $r-1$ dimensional subspace of $\R^n$. Let $\mathcal{L}$ be a $r-1-q$ dimensional subspace of $\mathcal{K}\cap (\text{Im} \ Q)^{\perp}$, and let $Q_{\mathcal{L}}$ be the orthogonal projection onto this subspace. 
    For $v\in \mathcal{K}$ we have
    $$\langle v,L_Hv\rangle = \sum\limits_{(i,j)\in E(H)}\norm{v_i-v_j}^2\ge \sum\limits_{(i,j)\in E(K)}\norm{v_i-v_j}^2=r\sum\limits_{i\in K}v_i^2-\norm{\sum\limits_{i\in K}v_i}^2=r\norm{v}^2.$$
    This in turn implies $$\tr(Q_{\mathcal{L}}L_H)\ge r\text{dim}\ \mathcal{L}=r\left (r-1-q\right). $$
    Since $Q+Q_{\mathcal L}$ is a projection of rank $r-1$, the trace of $L_H$ on its image is at most the sum of the largest $r-1$ eigenvalues of $L_H$. Thus, by Brouwer's inequality (\prettyref{thm:brouwer}), we obtain
    $$\tr\left ((Q+Q_{\mathcal{L}})L_H\right )\le \sum\limits_{i=1}^{r-1}\lambda_i(H)\le r(r-1)+ E(K,S), $$
    and the desired inequality follows.
\end{proof}

\begin{proof}[\prettyref{lem:traceresult} implies Grone-Merris-Bai]
     Order the vertices so that $d_1\ge d_2\ge \ldots \ge d_n$. We will show that
     $$\sum\limits_{i=1}^{n}(\lambda_i-r)_+\le \sum\limits_{i=r+1}^{n}d_i,$$
     for every $0\le r\le n$. The case $r=0$ follows from $\sum_i\lambda_i=2m=\sum_i d_i$. Let $1\le r\le n$.
     Let $Q$ be the projection on the eigenspace corresponding to the set of $>0$ eigenvalues of $L_G-rI_n$ so that
     $\sum\limits_{i=1}^{n}(\lambda_i-r)_+=\tr(Q(L_G-rI_n))$. 
     
     Let $H$ be the split graph obtained from $G$ by taking the set $K$ of the vertices with the largest $r$ degrees and making it a clique, taking the set $S$ of the vertices with the smallest $s=n-r$ degrees and making it an independent set, and keeping any edges in-between $K$ and $S$. We have
     $$\tr (QL_G)=\sum\limits_{(i,j)\in E(G)}\norm{Qe_i-Qe_j}^2\le 2e+\sum\limits_{(i,j)\in E(H)}\norm{Qe_i-Qe_j}^2=\tr(QL_H)+2e,$$
     where $e$ is the number of edges removed, i.e. the edges between the vertices with the smallest $s$ degrees. By \prettyref{lem:traceresult} we obtain
     $$\sum\limits_{i=1}^{n}(\lambda_i-r)_+=\tr(Q(L_G-rI_n))\le \tr(Q(L_H-rI_n))+2e\le E(K,S)+2e=\sum\limits_{i=r+1}^n d_i.$$
     We now use two elementary facts. First, $\sum_{i=1}^k\lambda_i\le rk+\sum_i(\lambda_i-r)_+$ for every $r$. Next, since the degrees are in the descending order, $\min_{0\le r\le n}\bigl(rk+\sum_{i=r+1}^n d_i\bigr)=\sum_i\min(d_i,k)$. We thus obtain the Grone-Merris-Bai inequality:
     $$\sum\limits_{i=1}^{k}\lambda_i\le \min_{0\le r\le n}\left (rk+\sum\limits_{i=1}^{n}(\lambda_i-r)_+\right )\le \min_{0\le r\le n}\left (rk+\sum\limits_{i=r+1}^{n}d_i\right )=\sum\limits_{i=1}^{k}\min(d_i,k)=\sum\limits_{i=1}^{k}d_i'.$$
\end{proof}
\section{Brouwer's Laplacian conjecture}
In this section, we will show that \prettyref{lem:nuclearnormresult} implies Brouwer's Laplacian conjecture. 
\begin{lemma}[Eigenvalue Sums Reduce to Projections]\label{lem:projectionreduction}
Let $G$ be a graph on $n$ vertices and $m$ edges which has Laplacian $L$ . Fix $1\le k\le n-1$. Suppose that every orthogonal projection $P$ of rank $k$ satisfying $P1=0$ obeys
\[
    \sum\limits_{(i,j)\in E(G)}(P_{ii}+P_{jj}-2P_{ij}-1)\le \dbinom{k+1}{2}.
\]
Then
\[
    \sum\limits_{i=1}^{k}\lambda_i(L)\le m+\dbinom{k+1}{2}.
\]
\end{lemma}
\begin{proof}
Since $L1=0$, the subspace $1^{\perp}$ is invariant under $L$ and the eigenvalues of the restriction of $L$ to $1^{\perp}$ are $\lambda_1,\ldots,\lambda_{n-1}$. Let $P$ be the projection matrix to the subspace spanned by the top $k$ eigenvectors of $L$. Then, $P$ is rank $k$ with $P1=0$ such that
\[
    \sum\limits_{i=1}^{k}\lambda_i(L)=\tr(PL).
\]
Using the decomposition
\[
    L=\sum\limits_{(i,j)\in E(G)}(e_i-e_j)(e_i-e_j)^T,
\]
we obtain
\[
    \tr(PL)
    =\sum\limits_{(i,j)\in E(G)}(e_i-e_j)^TP(e_i-e_j)
    =\sum\limits_{(i,j)\in E(G)}(P_{ii}+P_{jj}-2P_{ij}).
\]
The assumed projection estimate therefore gives
\[
    \sum\limits_{i=1}^{k}\lambda_i(L)
    \le m+\dbinom{k+1}{2}.
\]
\end{proof}

\begin{lemma}[A Projection Identity]\label{lem:keyidentityprop}
For any orthogonal projection $P$ of rank $k$ which satisfies $P1=0$, the following holds
    \begin{equation} \label{eq:keyidentity}
    \dfrac{1}{4}\sum\limits_{i\ne j}\left[ (P_{ii}+P_{jj}-2P_{ij})^2-(P_{ii}-P_{jj})^2\right]=k(k+1).
\end{equation}
\end{lemma}
\begin{proof}
Let $a_{ij}=P_{ii}+P_{jj}-2P_{ij}$. Since $P^2=P$, $\tr(P)=k$, and $P1=0$, we have
\[
\sum\limits_{i,j} a_{ij}^2
=2n\sum\limits_i P_{ii}^2+2k^2+4k.
\]
Indeed, the mixed term
\[
\sum\limits_{i,j}(P_{ii}+P_{jj})P_{ij}
\]
vanishes because every row and column sum of $P$ is zero, while $\sum_{i,j}P_{ij}^2=\tr(P^2)=k$. Also,
\[
\sum\limits_{i,j}(P_{ii}-P_{jj})^2
=2n\sum\limits_i P_{ii}^2-2k^2.
\]
Subtracting gives
\[
\sum\limits_{i,j}\left[a_{ij}^2-(P_{ii}-P_{jj})^2\right]=4k(k+1).
\]
The terms with $i=j$ vanish, so the desired identity follows.
\end{proof}

The following lemma is the key technical ingredient in our proof.

\begin{lemma}\label{lem:minequalityprop}
Let $P$ be an orthogonal projection of rank $k$ which satisfies $P1=0$. Define the matrix $M$ with diagonal entries $M_{ii}=0$ and off-diagonal entries $M_{ij}=P_{ii}+P_{jj}-2P_{ij}-1$. Let $v=CM1$. Then
\begin{equation}\label{eq:Minequality}
    \sum\limits_{i\ne j}(1-|M_{ij}|)^2\ge \dfrac{2}{n}\norm{v}^2.
\end{equation}
\end{lemma}

\begin{remark}[The cut/routing viewpoint]

Observe that if the left hand side has no absolute value on $M_{ij}$, then the lower bound above is an easy ``row-sum'' estimate. Indeed,  $W_{ij}=1-M_{ij}\quad (i\neq j),\qquad W_{ii}=0$, 
we have
$CW\mathbf 1=-CM\mathbf 1=-v$. 
Thus,
\[
\sum_{i\ne j}(1-M_{ij})^2
=
2\sum_{i<j}W_{ij}^2
\ge
\frac{2}{n}\|CW\mathbf 1\|^2
=
\frac{2}{n}\|v\|^2.
\]

The lemma shows that the \emph{same} lower bound remains true even after
replacing $1-M_{ij}$ by the smaller quantity $1-|M_{ij}|$. Our final proof is short but let us describe the intuition behind. 

First, observe that it is enough to prove (the formally stronger) claim that for every $x$, $|\langle x,v\rangle| \le \sum_{i<j}(1-|M_{ij}|) |x_i-x_j|$. The inequality we want then follows by plugging in $x= v$. This inequality has a pleasing and familiar looking form of ``capacitated routing'' with capacities $c_{ij} = (1-|M_{ij}|)$ on the complete graph (see, e.g.,~\cite[Theorem~5.19]{ChekuriNotes}). To verify the inequality for a given $x$, one need only check the ``level'' cuts: sort the vertices of $x$ and take all cuts formed by prefixes $U_1, U_2,\ldots,U_{n-1}$ in this ordering. We then observe that for any $r$, our proof reduces this problem to the nuclear norm estimate from \prettyref{lem:nuclearnormresult} for the split graph with a clique on $U_r$. This relies on the simple but key identity in ~\eqref{eq:prefixtraceidentity}. 
\end{remark}

\begin{proof}[Proof of \prettyref{lem:minequalityprop} via \prettyref{lem:final} below]
If $v=0$, the result is immediate. Let $v \neq 0$ in the following. 
Then, \prettyref{eq:Minequality} follows from \prettyref{lem:final} by Cauchy-Schwarz:
$$n\norm{v}^2\cdot  \sum\limits_{i< j}(1-|M_{ij}|)^2= \sum\limits_{i< j}(1-|M_{ij}|)^2\cdot\sum\limits_{i< j}|v_i-v_j|^2\ge \left ( \sum\limits_{i<j}(1-|M_{ij}|)\cdot |v_i-v_j|\right )^2\ge \norm{v}^4.$$

Now, observe that $\sum_{i<j} (v_i-v_j)^2 = n \norm{v}^2$. Dividing by $n\norm{v}^2$ gives
\[
        \sum_{i<j}(1-|M_{ij}|)^2\geq \frac{1}{n}\norm{v}^2,
\]
which is equivalent to \prettyref{eq:Minequality}.
\end{proof}

\begin{lemma}[A Routing Inequality for $M$]\label{lem:final}
Let $P$ be an orthogonal projection of rank $k$ which satisfies $P1=0$. Define the matrix $M$ with diagonal entries $M_{ii}=0$ and off-diagonal entries $M_{ij}=P_{ii}+P_{jj}-2P_{ij}-1$. Let $N=CMC$ and $v=CM1$.
For any $x\in \R^n$ the following inequality holds
$$|\langle x,v\rangle |\le \sum\limits_{i<j}(1-|M_{ij}|)\cdot |x_i-x_j|.$$
In particular,
$$ \sum\limits_{i<j}(1-|M_{ij}|)|v_i-v_j|\ge \norm{v}^2.$$
\end{lemma}

\begin{proof}
    For $i\ne j$, we have
    \[
        P_{ii}+P_{jj}-2P_{ij}=\norm{P(e_i-e_j)}^2\in [0,2],
    \]
    and hence $1-|M_{ij}|\ge 0$.
    We now compute the norm of $N$. Let $d=(P_{11},\ldots,P_{nn})^T$. Since
    \[
        M=d1^T+1d^T-2P-J+I,
    \]
    and $C1=0$, $CP=PC=P$, we get
    \[
        N=CMC=C-2P.
    \]
    Thus $N$ has eigenvalue $-1$ on $\operatorname{Im}P$, eigenvalue $1$ on
    $1^\perp\cap(\operatorname{Im}P)^\perp$, and eigenvalue $0$ on
    $\operatorname{span}\{1\}$. In particular, $\norm{N}=1$.

    It is enough to prove the one-sided estimate 
    $$\langle x,v\rangle \le \sum\limits_{i<j}(1-|M_{ij}|)\cdot |x_i-x_j|.$$
    By relabeling the coordinates, we may assume without loss of generality that $x_1\ge x_2\ge \ldots \ge x_n$.

    For each $1\le r\le n-1$, let $G_r$ be the split graph on $[n]$ obtained by making $\{1,2,\ldots,r\}$ a clique, $\{r+1,\ldots, n\}$ an independent set, and adding an edge across the cut wherever $M_{ij}\ge 0$. We first prove the identity
    \begin{equation}\label{eq:prefixtraceidentity}
    \sum\limits_{i=1}^{r}v_i+\sum\limits_{i=1}^{r}\sum\limits_{j=r+1}^{n} |M_{ij}|=-\tr\left ((L_{G_r}-rC)N\right ).
    \end{equation}
    Since $L_{G_r}1=C1=0$ and $N=CMC$, the right hand side is
    \[
        -\tr\left ((L_{G_r}-rC)M\right ).
    \]
    The diagonal entries of $M$ are zero, and therefore
    \[
        -\tr\left ((L_{G_r}-rC)M\right )
        =
        2\sum\limits_{ij\in E(G_r)}M_{ij}+r\tr(CM).
    \]
    Writing $K_r=\{1,\ldots,r\}$, the last expression is
    \[
        2\sum\limits_{\substack{i<j\\ i,j\in K_r}}M_{ij}
        +2\sum\limits_{\substack{i\le r<j\\ M_{ij}\ge 0}}M_{ij}
        -\dfrac{r}{n}1^TM1.
    \]
    On the other hand,
    \[
        \sum\limits_{i=1}^{r}v_i
        =
        2\sum\limits_{\substack{i<j\\ i,j\in K_r}}M_{ij}
        +\sum\limits_{i=1}^{r}\sum\limits_{j=r+1}^{n}M_{ij}
        -\dfrac{r}{n}1^TM1.
    \]
    Adding $\sum_{i=1}^{r}\sum_{j=r+1}^{n}|M_{ij}|$ and using
    $M_{ij}+|M_{ij}|=2(M_{ij})_+$ proves \prettyref{eq:prefixtraceidentity}.

    By \prettyref{lem:nuclearnormresult} and $\norm{N}=1$, \prettyref{eq:prefixtraceidentity} gives
    \[
    \sum\limits_{i=1}^{r}v_i+\sum\limits_{i=1}^{r}\sum\limits_{j=r+1}^{n} |M_{ij}|
    \le \norm{L_{G_r}-rC}_* \cdot \norm{N}\le r(n-r).
    \]
    Hence
    \begin{equation}\label{eq:prefixineq}
        \sum\limits_{i=1}^{r}v_i\le
        \sum\limits_{i=1}^{r}\sum\limits_{j=r+1}^{n}(1-|M_{ij}|).
    \end{equation}
    Since $\sum_i v_i=0$, telescoping gives
    \[
        \langle x,v\rangle
        =
        \sum\limits_{r=1}^{n-1}(x_r-x_{r+1})\sum\limits_{i=1}^{r}v_i.
    \]
    Combining this with \prettyref{eq:prefixineq}, and using $x_r-x_{r+1}\ge 0$, we obtain
    \[
        \langle x,v\rangle
        \le
        \sum\limits_{r=1}^{n-1}(x_r-x_{r+1})
        \sum\limits_{i=1}^{r}\sum\limits_{j=r+1}^{n}(1-|M_{ij}|)
        =
        \sum\limits_{i<j}(1-|M_{ij}|)(x_i-x_j).
    \]
    This is the desired one-sided estimate. Applying it to $-x$ gives the absolute-value form.
\end{proof}

\begin{lemma}\label{lem:strongmatrixineq}
For any orthogonal projection $P$ of rank $k$ which satisfies $P1=0$, the following holds
 \begin{equation}\label{eq:mainineq}
    \sum\limits_{i\ne j}(P_{ii}+P_{jj}-2P_{ij}-1)_+\le k(k+1).
 \end{equation}
In particular, \prettyref{lem:nuclearnormresult} implies Brouwer's Laplacian conjecture.
\end{lemma}
\begin{remark}
Observe that for any rank $k$ projection matrix $P$, the graph $G$ of pairs $i,j$ such that $i \neq j$ and $(P_{ii}+P_{jj}-2P_{ij}-1)>0$ gives the ``worst-case'' graph, i.e., maximizes the projection expression in~\prettyref{lem:projectionreduction}. 
\end{remark}
\begin{proof}
Using \prettyref{lem:keyidentityprop} and the fact that $(x-1)_+=\frac{1}{4}\left (x^2-(1-|x-1|)^2\right )$, inequality \prettyref{eq:mainineq} is equivalent to
\begin{equation}
\sum\limits_{i\ne j} \left ( 1-|P_{ii}+P_{jj}-2P_{ij}-1|\right )^2\ge \sum\limits_{i\ne j} (P_{ii}-P_{jj})^2. \label{eq:intermediate}
\end{equation}

Define the matrix $M$ with diagonal entries $M_{ii}=0$ and off-diagonal entries $M_{ij}=P_{ii}+P_{jj}-2P_{ij}-1$. Then,  $M\mathbf 1
=
n(P_{11},\ldots,P_{nn})^T+(k-n+1)\mathbf 1$, 
and hence $v=CM\mathbf 1=nC(P_{11},\ldots,P_{nn})^T$.
Thus,  $v_i-v_j=n(P_{ii}-P_{jj})$.

Finally, $v\perp \mathbf 1$,
\[
\sum_{i\ne j}(P_{ii}-P_{jj})^2
=
\frac1{n^2}\sum_{i\ne j}(v_i-v_j)^2
=
\frac2n\norm{v}^2.
\]
Thus the \prettyref{eq:intermediate} is exactly \prettyref{eq:Minequality}, which holds by \prettyref{lem:minequalityprop}.

The case $k=n$ is immediate from $\sum_{i=1}^{n}\lambda_i=2m$ and $m\le \binom{n}{2}\le \binom{n+1}{2}$. 

Now, let $1\le k\le n-1$. For any orthogonal projection $P$ of rank $k$ satisfying $P1=0$, \prettyref{eq:mainineq} gives
\begin{align*}
\sum\limits_{(i,j)\in E(G)}(P_{ii}+P_{jj}-2P_{ij}-1)
&\le \sum\limits_{(i,j)\in E(G)}(P_{ii}+P_{jj}-2P_{ij}-1)_+\\
&\le \dfrac{1}{2}\sum\limits_{i\ne j}(P_{ii}+P_{jj}-2P_{ij}-1)_+\\
&\le \dbinom{k+1}{2}.
\end{align*}

\end{proof}

Brouwer's Laplacian conjecture follows by combining \prettyref{lem:projectionreduction} with \prettyref{lem:strongmatrixineq}.

\section*{Statement of AI Use}
We used GPT 5.5 Pro to simplify our proof that Brouwer's Inequality implies the Grone-Merris-Bai theorem. We also used it for a proofreading of our manuscript. 

\bibliographystyle{alpha}
\bibliography{refs}

@article{Bai,
  title={{The Grone-Merris conjecture}},
  author={Bai, Hua},
  journal={Transactions of the American Mathematical Society},
  volume={363},
  number={8},
  pages={4463--4474},
  year={2011},
  doi={10.1090/S0002-9947-2011-05393-6},
  url={https://ams.org}
}

@book{BrouwerHaemers2012,
  title={Spectra of Graphs},
  author={Brouwer, Andries E. and Haemers, Willem H.},
  series={Universitext},
  publisher={Springer},
  address={New York},
  year={2012},
  doi={10.1007/978-1-4614-1939-6}
}

@article{HaemersMohammadianTayfehRezaie2010,
  title={On the sum of Laplacian eigenvalues of graphs},
  author={Haemers, Willem H. and Mohammadian, Ali and Tayfeh-Rezaie, Behruz},
  journal={Linear Algebra and its Applications},
  volume={432},
  number={9},
  pages={2214--2221},
  year={2010},
  doi={10.1016/j.laa.2009.03.038}
}

@article{FritscherHoppenRochaTrevisan2011,
  title={On the sum of the Laplacian eigenvalues of a tree},
  author={Fritscher, Eliseu and Hoppen, Carlos and Rocha, Israel and Trevisan, Vilmar},
  journal={Linear Algebra and its Applications},
  volume={435},
  number={2},
  pages={371--399},
  year={2011},
  doi={10.1016/j.laa.2011.01.036}
}

@article{DuZhou2012,
  title={Upper bounds for the sum of Laplacian eigenvalues of graphs},
  author={Du, Zhibin and Zhou, Bo},
  journal={Linear Algebra and its Applications},
  volume={436},
  number={9},
  pages={3672--3683},
  year={2012},
  doi={10.1016/j.laa.2012.01.007}
}

@article{WangHuangLiu2012,
  title={On a conjecture for the sum of Laplacian eigenvalues},
  author={Wang, Shouzhong and Huang, Yufei and Liu, Bolian},
  journal={Mathematical and Computer Modelling},
  volume={56},
  number={3--4},
  pages={60--68},
  year={2012},
  doi={10.1016/j.mcm.2011.12.047}
}

@mastersthesis{Mayank2010,
  title={On variants of the Grone--Merris conjecture},
  author={{Mayank}},
  school={Eindhoven University of Technology},
  address={Eindhoven},
  year={2010}
}

@article{HelmbergTrevisan2017,
  title={Spectral threshold dominance, {Brouwer}'s conjecture and maximality of Laplacian energy},
  author={Helmberg, Christoph and Trevisan, Vilmar},
  journal={Linear Algebra and its Applications},
  volume={512},
  pages={18--31},
  year={2017},
  doi={10.1016/j.laa.2016.09.029}
}

@article{GanieAlghamdiPirzada2016,
  title={On the sum of the Laplacian eigenvalues of a graph and {Brouwer}'s conjecture},
  author={Ganie, Hilal A. and Alghamdi, A. M. and Pirzada, S.},
  journal={Linear Algebra and its Applications},
  volume={501},
  pages={376--389},
  year={2016},
  doi={10.1016/j.laa.2016.03.034}
}

@article{Chen2018,
  title={Improved results on {Brouwer}'s conjecture for sum of the Laplacian eigenvalues of a graph},
  author={Chen, Xiaodan},
  journal={Linear Algebra and its Applications},
  volume={557},
  pages={327--338},
  year={2018},
  doi={10.1016/j.laa.2018.08.003}
}

@article{Chen2019,
  title={On {Brouwer}'s conjecture for the sum of $k$ largest Laplacian eigenvalues of graphs},
  author={Chen, Xiaodan},
  journal={Linear Algebra and its Applications},
  volume={578},
  pages={402--410},
  year={2019},
  doi={10.1016/j.laa.2019.05.029}
}

@article{GaniePirzadaRatherTrevisan2020,
  title={Further developments on {Brouwer}'s conjecture for the sum of Laplacian eigenvalues of graphs},
  author={Ganie, Hilal A. and Pirzada, S. and Rather, Bilal A. and Trevisan, Vilmar},
  journal={Linear Algebra and its Applications},
  volume={588},
  pages={1--18},
  year={2020},
  doi={10.1016/j.laa.2019.11.020}
}

@article{Cooper2021,
  title={Constraints on {Brouwer}'s Laplacian spectrum conjecture},
  author={Cooper, Joshua N.},
  journal={Linear Algebra and its Applications},
  volume={615},
  pages={11--27},
  year={2021},
  doi={10.1016/j.laa.2020.12.028}
}

@article{TorresTrevisan2024,
  title={{Brouwer}'s conjecture for the Cartesian product of graphs},
  author={Torres, Guilherme S. and Trevisan, Vilmar},
  journal={Linear Algebra and its Applications},
  volume={685},
  pages={66--76},
  year={2024},
  doi={10.1016/j.laa.2023.12.019}
}

@misc{LewPartitionDensity2024,
  title={Partition density, star arboricity, and sums of Laplacian eigenvalues of graphs},
  author={Lew, Alan},
  year={2024},
  eprint={2410.04563},
  archivePrefix={arXiv},
  primaryClass={math.CO},
  doi={10.48550/arXiv.2410.04563}
}

@misc{LewSumsDegrees2025,
  title={Sums of Laplacian eigenvalues and sums of degrees},
  author={Lew, Alan},
  year={2025},
  eprint={2508.04209},
  archivePrefix={arXiv},
  primaryClass={math.CO},
  doi={10.48550/arXiv.2508.04209}
}

@misc{LewApproximate2026,
  title={An approximate version of {Brouwer}'s Laplacian conjecture},
  author={Lew, Alan},
  year={2026},
  eprint={2601.17575},
  archivePrefix={arXiv},
  primaryClass={math.CO},
  doi={10.48550/arXiv.2601.17575}
}

@misc{BakshiBasuKothariLi2026,
  title={Sharp Bounds on the Eigenvalues of {Kikuchi} Graphs and Applications to Quantum Max Cut},
  author={Bakshi, Ainesh and Basu, Arpon and Kothari, Pravesh K. and Li, Anqi},
  year={2026},
  eprint={2605.14994},
  archivePrefix={arXiv},
  primaryClass={quant-ph},
  doi={10.48550/arXiv.2605.14994}
}

@article{GroneMerrisSunder1990,
  title={The Laplacian spectrum of a graph},
  author={Grone, Robert and Merris, Russell and Sunder, V. S.},
  journal={SIAM Journal on Matrix Analysis and Applications},
  volume={11},
  number={2},
  pages={218--238},
  year={1990},
  doi={10.1137/0611016}
}

@article{Merris1994,
  title={Degree maximal graphs are Laplacian integral},
  author={Merris, Russell},
  journal={Linear Algebra and its Applications},
  volume={199},
  pages={381--389},
  year={1994},
  doi={10.1016/0024-3795(94)90361-1}
}

@article{GroneMerris1994,
  title={The Laplacian spectrum of a graph {II}},
  author={Grone, Robert and Merris, Russell},
  journal={SIAM Journal on Discrete Mathematics},
  volume={7},
  number={2},
  pages={221--229},
  year={1994},
  doi={10.1137/S0895480191222653}
}

@article{HammerKelmans1996,
  title={Laplacian spectra and spanning trees of threshold graphs},
  author={Hammer, Peter L. and Kelmans, Alexander K.},
  journal={Discrete Applied Mathematics},
  volume={65},
  number={1--3},
  pages={255--273},
  year={1996},
  doi={10.1016/0166-218X(94)00049-J}
}

@article{DuvalReiner2002,
  title={Shifted simplicial complexes are Laplacian integral},
  author={Duval, Art M. and Reiner, Victor},
  journal={Transactions of the American Mathematical Society},
  volume={354},
  number={11},
  pages={4313--4344},
  year={2002},
  doi={10.1090/S0002-9947-02-03082-9}
}

@article{Stephen2007,
  title={A majorization bound for the eigenvalues of some graph Laplacians},
  author={Stephen, Tamon},
  journal={SIAM Journal on Discrete Mathematics},
  volume={21},
  number={2},
  pages={303--312},
  year={2007},
  doi={10.1137/040619594}
}

@article{AbiadFiolHaemersPerarnau2014,
  title={An interlacing approach for bounding the sum of Laplacian eigenvalues of graphs},
  author={Abiad, Aida and Fiol, Miquel A. and Haemers, Willem H. and Perarnau, Guillem},
  journal={Linear Algebra and its Applications},
  volume={448},
  pages={11--21},
  year={2014},
  doi={10.1016/j.laa.2014.02.003}
}

@misc{ChekuriNotes,
  title={{CS 586/IE 519: Combinatorial Optimization}},
  author={Chekuri, Chandra},
  howpublished={Lecture notes, University of Illinois Urbana--Champaign},
  year={2022},
  note={See Theorem 5.19},
  url={https://courses.engr.illinois.edu/cs586/sp2022/main.pdf}
}
\end{document}